\documentclass{amsart}

\usepackage{amsmath, amssymb, amsfonts, amsthm}
\usepackage{asymptote}
\usepackage{enumerate}
\usepackage{versions}
\usepackage[utf8]{inputenc}
\usepackage{fancyhdr}
\usepackage{etoolbox}
\usepackage{extarrows}
\usepackage{mathtools}
\usepackage[colorlinks]{hyperref}
\usepackage{relsize}
\usepackage{thm-restate}

\newcommand{\NN}{\mathbb N}

\newtheorem{theorem}{Theorem}[section]
\newtheorem{lemma}[theorem]{Lemma}
\newtheorem{corollary}[theorem]{Corollary}

\newtheorem{conjecture}[theorem]{Conjecture}

\theoremstyle{definition}

\newtheorem{remark}[theorem]{Remark}

\usepackage[foot]{amsaddr}

\newtheorem{prob}[theorem]{Problem}

\newcommand{\I}{\mathcal I}
\newcommand{\J}{\mathcal J}
\renewcommand{\S}{\mathfrak S}

\DeclareMathOperator{\des}{des}
\DeclareMathOperator{\dd}{dd}
\DeclareMathOperator{\DES}{Des}
\DeclareMathOperator{\DD}{DD}
\DeclareMathOperator{\maj}{maj}

\title{The Eulerian distribution on involutions is indeed $\gamma$-positive}
\author[Danielle Wang]{Danielle Wang}
\address{Department of Mathematics, Massachusetts Institute of Technology, Cambridge, MA 02139-4307}
\email{diwang@mit.edu}

\begin{document}
\begin{abstract}
	Let $\I_n$ and $\J_n$ denote the set of involutions
	and fixed-point free involutions of
	$\{1, \dots, n\}$, respectively, and let
	$\des(\pi)$ denote the number of descents
	of the permutation $\pi$. We prove a conjecture of Guo and Zeng
	which states that 
	$I_n(t) \coloneqq \sum_{\pi \in \I_n} t^{\des(\pi)}$ 
	is $\gamma$-positive for $n \ge 1$ and 
	$J_{2n}(t) \coloneqq \sum_{\pi \in \J_{2n}} t^{\des(\pi)}$ 
	is $\gamma$-positive for $n \ge 9$.
	We also prove that the number of 
	$(3412, 3421)$-avoiding permutations with $m$
	double descents and $k$ descents is equal to the
	number of separable permutations with $m$ double
	descents and $k$ descents.
\end{abstract}

\maketitle

\section{Introduction}
A polynomial $p(t) = a_rt^r + a_{r+1}t^{r+1}
+ \cdots + a_st^s$ is called \emph{palindromic of
center} $\frac n2$ if $n = r + s$ and
$a_{r+i} = a_{s-i}$ for $0 \le i \le \frac n2 - r$. 
A palindromic polynomial can be written uniquely 
\cite{branden2015unimodality} as
\[
	p(t) = \sum_{k=r}^{\lfloor \frac n2\rfloor}
	\gamma_k t^k(1 + t)^{n - 2k},
\]
and it is called $\gamma$-positive if 
$\gamma_k \ge 0$ for each $k$. The $\gamma$-positivity
of a palindromic polynomial implies unimodality
of its coefficients (i.e., the coefficients $a_i$
satisfy $a_r \le a_{r+1} \le \cdots \le a_{\lfloor n/2\rfloor} \ge a_{\lfloor n/2\rfloor + 1} \ge \cdots \ge a_s$).

Let $\S_n$ be the set of all permutations of
$[n] = \{1, 2, \dots, n\}$.
For $\pi \in \S_n$, the \emph{descent set}
of $\pi$ is
\[
	\DES(\pi) = \{i \in [n-1] : \pi(i) > \pi(i+1)\},
\]
and the \emph{descent number} is
$\des(\pi) = \#\DES(\pi)$.
The \emph{double descent set} is
\[
	\DD(\pi) = \{ i \in [n] : \pi(i-1) > \pi(i) > \pi(i+1) \}
\]
where $\pi(0) = \pi(n+1) = \infty$, and we define
$\dd(\pi) = \#\DD(\pi)$.

Finally, a permutation $\pi$ is said to
\emph{avoid} a permutation $\sigma$ (henceforth called a \emph{pattern}) if $\pi$ does not contain a
subsequence (not necessarily consecutive) with the
same relative order as $\sigma$.
We let $\S_n(\sigma_1, \dots, \sigma_r)$ denote
the set of permutations in $\S_n$ avoiding the patterns
$\sigma_1,\dots, \sigma_r$.

The descent polynomial $A_n(t) = \sum_{\pi \in \S_n} t^{\des(\pi)}$ is called the \emph{Eulerian polynomial},
and we have the following remarkable fact, which
implies that $A_n(t)$ is $\gamma$-positive.

\begin{theorem}[Foata--Sch\"utzenberger \cite{foata2006theorie}]
\label{thm:foata-schutz}
For $n \ge 1$,
\[
	A_n(t) = \sum_{k = 0}^{\lfloor \frac{n-1}{2} \rfloor}
	\gamma_{n,k}t^k(1+t)^{n-1-2k},
\]
where $\gamma_{n,k} = \# \{ \pi \in \S_n :
\dd(\pi) = 0, \des(\pi) = k\}$.
\end{theorem}

Similarly, let $\I_n$ be the set of all involutions
in $\S_n$, and let $\J_n$ be the set
of all fixed-point free involutions in $\S_n$.
Define 
\[
	I_n(t) = \sum_{\pi \in \I_n} t^{\des(\pi)}, \quad
	J_n(t) = \sum_{\pi \in \J_n} t^{\des(\pi)}.
\]
Note that $J_n(t) = 0$ for $n$ odd.
Strehl \cite{strehl1981symmetric} first showed that
the polynomials $I_n(t)$ and $J_{2n}(t)$ are palindromic.
Guo and Zeng \cite{guo2006eulerian} proved that 
$I_n(t)$ and $J_{2n}(t)$ are unimodal
and conjectured that they are in fact $\gamma$-positive.
Our first two theorems, which we prove in
Sections \ref{sec:In} and
\ref{sec:Jn}, confirm their conjectures.

\begin{restatable}{thmm}{In}
\label{thm:In}
	For $n \ge 1$, the polynomial $I_n(t)$ is $\gamma$-positive.
\end{restatable}

\begin{restatable}{thmm}{Jn}
\label{thm:Jn}
	For $n \ge 9$, the polynomial $J_{2n}(t)$ is $\gamma$-positive.
\end{restatable}

Theorem \ref{thm:foata-schutz} and many variations
of it have been proved by many methods, see for
example \cite{lin2016proof,shin2012symmetric}.
One of these methods uses the \emph{Modified Foata--Strehl (MFS) action} 
on $\S_n$
\cite{branden2008actions,lin2015gamma,postnikov2008faces}.
In fact it follows from \cite[Theorem 3.1]{branden2008actions} that the same property holds
for all subsets of $\S_n$ which are invariant under
the MFS action. The $(2413, 3142)$-avoiding permutations
are the \emph{separable permutations}, which
are permutations that can be built from the 
trivial permutation through \emph{direct sums} 
and \emph{skew sums} \cite[Theorem 2.2.36]{kitaev2011patterns}.
These are not invariant under the MFS action.
However, in 2017, Fu, Lin, and Zeng \cite{fu2018two},
using a bijection with di-sk trees,
and Lin \cite{lin2017gamma}, using an algebraic approach,
proved that the separable permutations also satisfy 
the following theorem.

\begin{theorem}[{\cite[Theorem 1.1]{fu2018two}}]
	For $n \ge 1$,
	\[
		\sum_{\pi \in \S_n(2413, 3142)} t^{\des(\pi)} = \sum_{k = 0}^{\lfloor \frac{n-1}{2}\rfloor} \gamma_{n,k}^S t^k(1+t)^{n-1-2k},
	\]
	where $\gamma_{n,k}^S = \#\{ \pi \in \S_n(2413, 3142) : \dd(\pi) = 0, \des(\pi) = k\}$.
\end{theorem}

Two sets of patterns $\Pi_1$ and $\Pi_2$ are \emph{$\des$-Wilf} equivalent if 
\[
	\sum_{\pi \in \S_n(\Pi_1)} t^{\des(\pi)} = 
	\sum_{\pi \in \S_n(\Pi_2)} t^{\des(\pi)},
\]
and are \emph{$\DES$-Wilf} equivalent if
\[
	\sum_{\pi \in \S_n(\Pi_1)} \prod_{i \in \DES(\pi)}t_i
	= \sum_{\pi \in \S_n(\Pi_2)} \prod_{i \in \DES(\pi)}t_i.
\]
We also say that the permutation classes
$\S_n(\Pi_1)$ and $\S_n(\Pi_2)$ are $\des$-Wilf 
or $\DES$-Wilf equivalent.

In 2018, Lin and Kim \cite[Theorem 5.1]{lin2018sextuple}
determined all permutation classes avoiding two patterns of length $4$
which are $\des$-Wilf equivalent to the separable 
permutations, all of which are $\DES$-Wilf equivalent
to each other but not to the separable permutations.

One such class is the $(3412, 3421)$-avoiding
permutations, which is invariant under the MFS action. 
A byproduct
of this is that the number of $(3412, 3421)$-avoiding
permutations with no double descents and $k$ descents
is also equal to $\gamma_{n,k}^S$.
In Section \ref{sec:dddes}, we prove the following more general fact.

\begin{restatable}{thmm}{dddes}
\label{thm:dddes}
	For $n \ge 1$,
	\[
		\sum_{\pi \in \S_n(3412, 3421)} x^{\des(\pi)}y^{\dd(\pi)} = \sum_{\pi \in \S_n(2413, 3142)} x^{\des(\pi)}y^{\dd(\pi)}.
	\]
\end{restatable}

\section{Proof of the $\gamma$-positivity of $I_n(t)$}
\label{sec:In}
In this section we prove Theorem \ref{thm:In},
restated below for clarity. 
Let the $\gamma$-expansion of $I_n(t)$ be
\[
	I_n(t) = \sum_{k=0}^{\lfloor \frac{n-1}{2}\rfloor} 
	a_{n,k} t^k(1+t)^{n-2k-1}.
\]
We have the following recurrence relation for the 
coefficients $a_{n,k}$.
\begin{theorem}[{\cite[Theorem 4.2]{guo2006eulerian}}]
\label{thm:Inrecur}
For $n \ge 3$ and $k \ge 0$,
\begin{align*}
	na_{n,k} = &(k+1)a_{n-1,k} + (2n-4k)a_{n-1,k-1}
	+ [k(k+2) + n-1]a_{n-2,k} \\
	&+ [(k-1)(4n-8k-14) + 2n-8]a_{n-2,k-1} \\
	&+ 4(n-2k)(n-2k+1)a_{n-2,k-2},
\end{align*}
where $a_{n,k} = 0$ if $k < 0$ or $k > (n-1)/2$.
\end{theorem}

\In*
\begin{proof}
We will prove by induction on $n$ the slightly 
stronger claim that
$a_{n,k} \ge 0$ for $n \ge 1$, $k \ge 0$,
and $a_{n,k} \ge \frac 2n a_{n-1,k-1}$ if $n = 2k + 1$
and $k \ge 4$. Assume the claim is true whenever the 
first index is less than $m$.
We want to prove the claim for all $a_{m,k}$.
If $m \le 2000$, we can check the claim
directly (this has been done by the author using Sage). Thus, we may assume that $m > 2000$.
If $m \ge 2k + 3$, then all the coefficients
in the recursion are nonnegative, so we are done
by induction. Thus, assume that
$(m, k) = (2n+1, n)$ or $(2n+2, n)$ with $n \ge 1000$.
\\\\
\underline{Case 1:} $(m,k) = (2n+2, n)$.
We wish to show that $a_{2n+2,n} \ge 0$. We apply
the recurrence in Theorem \ref{thm:Inrecur}, noting that
$a_{2n, n} = 0$ since $n > (2n-1)/2$, to get
\begin{align*}
	(2n+2)a_{2n+2, n} &= (n+1)a_{2n+1,n} + 4a_{2n+1,n-1}
	+ 24a_{2n,n-2} - (2n-2)a_{2n,n-1} \\
	&\ge 4a_{2n+1,n-1} + 24a_{2n,n-2} - 2na_{2n,n-1} \\
	&\ge 4a_{2n+1, n-1} + 24a_{2n,n-2} - na_{2n-1, n-1}
	- 4a_{2n-1, n-2} \\
	& \quad - 24a_{2n-2, n-3}, \tag{$\dagger$}
\end{align*}
where the last inequality comes from applying the
recurrence once again to obtain
\[
	2n a_{2n,n-1} \le na_{2n-1, n-1} + 4a_{2n-1, n-2}
	+ 24a_{2n-2, n-3}.
\]
Note that since $2n+1 \ge 2(n-1) + 3$, when we
apply the recurrence relation for $a_{2n+1, n-1}$
all terms in the sum are positive. We drop all terms but the
$a_{2n-1, n-1}$ term and multiply by $4/(2n+1)$ to
get
\begin{align*}
	4a_{2n+1,n-1} &\ge \frac{4}{2n+1} \left[ (n-1)(n+1) + 2n \right] a_{2n-1,n-1} \ge na_{2n-1, n-1}.
\end{align*}
Similarly, since $2n \ge 2(n-2) + 3$, when we
use the recursion to calculate $a_{2n, n-2}$, we
can drop all terms but the $a_{2n-1, n-2}$ term,
which after multiplying by $12/(2n)$ gives
\begin{align*}
	12 a_{2n, n-2} \ge \frac{12}{2n}[(n-1) a_{2n-1, n-2}]
	\ge 4 a_{2n-1, n-2}.
\end{align*}
Alternatively, we could have dropped all but the
$a_{2n-2, n-3}$ term to get
\begin{align*}
	12a_{2n, n-2} &\ge \frac{12}{2n}[ (n-3)\cdot 2 + 4n - 8]a_{2n-2, n-3} \ge 24a_{n-2, n-3}.
\end{align*}
Plugging the previous three inequalities into
$(\dagger)$ proves that $a_{2n+2, n} \ge 0$, as desired.
\\\\
\underline{Case 2:} $(m,k) = (2n+1, n)$.
We want to show $(2n+1)a_{2n+1,n} \ge 2a_{2n, n-1}$.
By the recurrence relation, we have
\begin{align*}
	(2n+1)a_{2n+1,n} = 2a_{2n, n-1} + 8a_{2n-1, n-2}
	- (6n-4)a_{2n-1, n-1}.
\end{align*}
Thus, it suffices to show that 
$8a_{2n-1,n-2} - (6n-3)a_{2n-1,n-1} \ge 0$.
Note that
\begin{align*}
	8a_{2n-1,n-2} - &(6n-3)a_{2n-1,n-1} \\ &\ge
	8a_{2n-1,n-2} - 6a_{2n-2, n-2} - 24a_{2n-3, n-3}
	\tag{$\ast$}
\end{align*}
because, by applying the same recurrence relation for
$a_{2n-1, n-1}$ and dropping the 
$-(6n-10)a_{2n-3, n-2}$ term, which is negative, we see that 
\begin{align*}
	(6n-3)a_{2n-1,n-1} &= 3(2n-1)a_{2n-1, n-1}\\
	&\le 3(2a_{2n-2, n-2} + 8a_{2n-3, n-3})\\
	&= 6a_{2n-2, n-2} + 24a_{2n-3, n-3}.
\end{align*}
Multiplying the right hand side of $(\ast)$ by
$2n-1$ and applying the recurrence relation
for $a_{2n-1, n-2}$ we get
\begin{align*}
	(2n-1) \cdot (\ast) &= 8(2n-1)a_{2n-1, n-2} - (12n-6)a_{2n-2, n-2}
	- (48n - 24)a_{2n-3, n-3} \\
	&= 8[(n-1)a_{2n-2, n-2} + 6a_{2n-2, n-3} + (n^2 - 2)
	a_{2n-3, n-2}\\&\quad + (2n-4) a_{2n-3, n-3} + 48a_{2n-3, n-4}
	] \\&\quad - (12n - 6)a_{2n-2, n-2} - (48n - 24)a_{2n-3, n-3} \\
	&= 48a_{2n-2,n-3} + (8n^2-16)a_{2n-3,n-2}
	+ 384a_{2n-3,n-4}\\&\quad - (4n+2)a_{2n-2,n-2} - (32n+8)a_{2n-3,n-3}. \tag{$\ast\ast$}
\end{align*}
Now, since $2n-2\ge 2(n-3) + 3$, we use the recursion
to calculate $a_{2n-2, n-3}$, drop some terms,
and multiply by $48/(2n-2)$ to get
\begin{align*}
	48 a_{2n-2,n-3} \ge \frac{48}{2n-2}
	[(n-4)\cdot 2 + 4n-12] a_{2n-4,n-4}
	\ge 120 a_{2n-4, n-4}.
\end{align*}
Now we apply the recursion for $a_{2n-2, n-2}$ and
multiply by $(4n+2)/(2n-2)$, which is less than $5$, to obtain
\begin{align*}
	(4n + 2)a_{2n-2, n-2}& = \frac{4n+2}{2n-2} 
	[(n-1)a_{2n-3,n-2} + 4a_{2n-3,n-3} \\ &\quad
	+ 24a_{2n-4, n-4} - (2n-6)a_{2n-4, n-3}] \\
	&< (5n-5)a_{2n-3,n-2} + 20a_{2n-3, n-3} + 120a_{2n-4, n-4} \\
	&\quad- (10n-30)a_{2n-4,n-3}.
\end{align*}
We substitute the above two bounds on
$48a_{2n-2, n-3}$ and $(4n+2)a_{2n-2, n-2}$ for
the corresponding terms in $(\ast\ast)$ to get
\begin{align*}
	(\ast\ast) &\ge 120 a_{2n-4, n-4} + (8n^2-16)a_{2n-3,n-2}
	+ 384a_{2n-3,n-4}\\
	&\quad - (5n-5)a_{2n-3,n-2} - 20a_{2n-3, n-3} - 120a_{2n-4, n-4} \\
	&\quad + (10n-30)a_{2n-4,n-3} - (32n + 8)a_{2n-3, n-3} \\
	&= (8n^2 - 5n -11)a_{2n-3, n-2} + 384 a_{2n-3, n-4}
	+ (10n- 30)a_{2n-4, n-3} \\
	&\quad - (32n + 28)a_{2n-3, n-3}. \tag{$\ast\ast\ast$}
\end{align*}
Since $2n-3 = 2(n-2) + 1$, by the
$a_{2k+1, k} \ge \frac{2}{2k+1}a_{2k, k-1}$ part
of the induction hypothesis,
we have
\begin{align*}
	(8n^2-5n-11)a_{2n-3,n-2} \ge
	\frac{2(8n^2-5n-11)}{2n-3}a_{2n-4, n-3}
	\ge (8n + 6) a_{2n-4, n-3}.
\end{align*}
Plugging the previous inequality into $(\ast\ast\ast)$ gives
\begin{align*}
	(\ast\ast\ast) &\ge (18n - 24)a_{2n-4, n-3}
	+ 384a_{2n-3, n-4} - (32n + 28)a_{2n-3,n-3}.
\end{align*}
Since $2n-3 \ge 2(n-4) \ge 3$,
when apply the recursion for $a_{2n-3, n-4}$, we can
drop the $a_{2n-5, n-6}$ term which is positive, to
get (after multiplying by $384/(2n-3)$),
\begin{align*}
	384a_{2n-3,n-4} &\ge \frac{384}{2n-3} [ 
	(n-3)a_{2n-4,n-4} + 10a_{2n-4,n-5} \\
	&\quad + (n^2-4n + 4)a_{2n-5, n-4} + (10n-44)a_{2n-5,n-5}
	].
\end{align*}
Apply the recursion for $a_{2n-3, n-3}$ directly
and multiply by $32n+28/(2n-3)$ to get
\begin{align*}
	(32n + 28)a_{2n-3,n-3}&= \frac{32n+28}{2n-3}
	[(n-2)a_{2n-4,n-3} + 6a_{2n-4,n-4} \\
	&\quad + (n^2-2n-1) a_{2n-5,n-3} + (2n-6)a_{2n-5,n-4}\\
	&\quad + 48a_{2n-5, n-5} ].
\end{align*}
Now, we will check that each of the terms in the expansion
for $(32n + 28)a_{2n-3, n-3}$ is less than one of the
terms in the expansion of $(18n-24)a_{2n-4,n-3}
+ 384a_{2n-3,n-4}$.

We have $(32n+28)/(2n-3) \le 17$, and we see that
\begin{align*}
	17\cdot (n-2)a_{2n-4,n-3} &\le (18n-24)a_{2n-4,n-3}, \\
	17\cdot 6a_{2n-4,n-4} &\le \frac{384(n-3)}{2n-3}a_{2n-4,n-4},\\
	17\cdot (2n-6)a_{2n-5,n-4} &\le \frac{68(n^2-4n+4)}{2n-3}a_{2n-5,n-4}, \\
	17\cdot 48a_{2n-5,n-5} &\le  \frac{384(10n-44)}{2n-3}a_{2n-5,n-5}.
\end{align*}
Now, it suffices to show 
\begin{align*}
	17\cdot(n^2 - 2n - 1)a_{2n-5,n-3}
	\le \frac{316(n^2-4n+4)}{2n-3}a_{2n-5,n-4}.
\end{align*}
It suffices to show that 
\[
	9a_{2n-5,n-4} \ge na_{2n-5,n-3}.
\]
By the recurrence relation we have
\begin{align*}
	na_{2n-5,n-3} &\le \frac{n}{(2n-5)}(2a_{2n-6,n-4} + 8a_{2n-7,n-5}) \\
	4.5 a_{2n-5,n-4} &\ge \frac{4.5(n-3)}{(2n-5)}a_{2n-6,n-4} \\
	4.5 a_{2n-5,n-4} &\ge \frac{4.5(2n-8)}{(2n-5)} a_{2n-7,n-5}.
\end{align*}
Combining these gives the desired inequality.
\end{proof}

\section{Proof of the $\gamma$-positivity of $J_{2n}(t)$}
\label{sec:Jn}
In this section we prove Theorem \ref{thm:Jn},
restated below.
Let the $\gamma$-expansion of $J_{2n}(t)$ be
\[
	J_{2n}(t) = \sum_{k=1}^n b_{2n,k} t^k(1+t)^{2n-2k}.
\]
We have the following recurrence relation for the
coefficients $b_{2n,k}$.

\begin{theorem}[{\cite[Theorem 4.4]{guo2006eulerian}}]
\label{thm:Jnrecur}
	For $n\ge 2$ and $k\ge 1$, we have
	\begin{align*}
		2nb_{2n,k} &= [k(k+1) + 2n-2]b_{2n-2,k} + 
		[2+2(k-1)(4n-4k-3)]b_{2n-2,k-1} \\
		&\quad + 8(n-k+1)(2n-2k+1)b_{2n-2,k-2},
	\end{align*}
	where $b_{2n,k} = 0$ if $k < 1$ or $k > n$.
\end{theorem}

\Jn*
\begin{proof}
	We will prove by induction on $n$
	the slightly stronger claim that
	for $b_{2n,k} \ge 0$ for $n \ge 9, k \ge 1$,
	and $b_{2n, n} \ge b_{2n-2,n-1}$ for
	$n \ge 11$.
	Assume the claim is true whenever the first index
	is less than $m$. We want to prove the claim for
	all $b_{m, k}$. If $m \le 2000$, we can
	check the claim directly (this has been checked
	using Sage). Thus, we
	may assume $m > 2000$. If $m > 2k$, then all of the
	coefficients in the recursion are nonnegative,
	so we are done by induction. Thus, we can assume
	that $(m,k) = (2n, n)$ with $n > 1000$.

	By the recurrence relation, we have
	\begin{align*}
		2nb_{2n,n} = 8b_{2n-2,n-2} - (6n-8)b_{2n-2,n-1}.
	\end{align*}
	We want to show that 
	$8b_{2n-2,n-2} - (8n-8)b_{2n-2,n-1} \ge 0$.
	We have
	\begin{align*}
		8b_{2n-2,n-2} - &(8n-8)b_{2n-2,n-1} \\
		&= 8b_{2n-2,n-2} - 32b_{2n-4, n-3}
		 + 4(6n-14)b_{2n-4, n-2}.
	\end{align*}
	Multiplying by $(2n-2)/8$, it suffices to show
	\[
		(2n-2) b_{2n-2,n-2} - (8n-8) b_{2n-4,n-3} + (6n^2 - 20n + 14)b_{2n-4,n-2} \ge 0.
	\]
	By expanding $(2n-2)b_{2n-2, n-2}$ using the recursion,
	we find that the above is equivalent to
	\begin{align*}
		&(7n^2 - 21n + 12)b_{2n-4, n-2}
		+ 48b_{2n-4, n-4} - (6n-4)b_{2n-4,n-3}
		 \ge 0. \tag{$\ddagger$}
	\end{align*}
	By the induction hypothesis,
	\[
		(7n^2-21n+12)b_{2n-4,n-2} \ge (7n^2-21n+12)b_{2n-6,n-3}. 
	\]
	By the recurrence in Theorem \ref{thm:Jnrecur},
		\begin{align*}
		(2n-4)_{2n-4,n-4} &\ge
		 (n^2 - 5n + 6)b_{2n-6, n-4}
		+ (10n - 48)b_{2n-6,n-5}.
	\end{align*}
	Multiplying by $48/(2n-4)$ yields
	\begin{align*}
		48b_{2n-4,n-4} &\ge
		\frac{48}{2n-4}[ (n^2 - 5n + 6)b_{2n-6, n-4}
		+ (10n - 48)b_{2n-6,n-5} ].
	\end{align*}
	Also, multiplying the recurrence for $b_{2n-4, n-3}$
	by $(6n-4)/(2n-4)$ yields
	\begin{align*}
		(6n-4)b_{2n-4,n-3} &= \frac{6n-4}{2n-4}
		[(n^2 -3n)b_{2n-6,n-3}
		+ (2n - 6)b_{2n-6,n-4} \\
		&\qquad\qquad\qquad+ 48b_{2n-6,n-5} ].
	\end{align*}
	We check that each term in this sum is less that
	one of the terms in the expansion of
	$(7n^2-21n+12)b_{2n-6,n-3} + 48b_{2n-4,n-4}$. 
	We have $(6n-4)/(2n-4) \le 4$ and
	\begin{align*}
		4(n^2-3n)b_{2n-6,n-3} &\le (7n^2-21n+12)b_{2n-6,n-3} \\
		4(2n-6)b_{2n-6,n-4} &\le \frac{48(n^2-5n+6)}{2n-4}b_{2n-6,n-4} \\
		4\cdot 48 b_{2n-6,n-5} &\le
		\frac{48(10n-48)}{2n-4}b_{2n-6,n-5}.
	\end{align*}
	Thus $(\ddagger)$ is true, as desired.
\end{proof}

\section{$(3412, 3421)$-avoiding permutations and
separable permutations}
\label{sec:dddes}

In this section we prove Theorem \ref{thm:dddes}, restated below.

\dddes*
For convenience we define the following variants of the
double descent set.
Let
\[
	\DD_0(\pi) = \{ i \in [n] : \pi(i-1) > \pi(i) > \pi(i+1)\}
\]
where $\pi(0) = 0$, $\pi(n+1) = \infty$, and
\[
	\DD_\infty(\pi) = \{ i \in [n] : \pi(i-1) > \pi(i) > \pi(i+1)\}
\]
where $\pi(0) = \infty$, $\pi(n+1) = 0$.
Similarly define $\dd_0(\pi)$ and $\dd_\infty(\pi)$.
Finally, let
\[
	\des'(\pi) = \#(\DES(\pi) \setminus \{n-1\}),
	\quad \dd'(\pi) = \#(\DD(\pi) \setminus \{n - 1\}).
\]
Let $\S_n^1 = \S_n(2413, 3142)$ and $\S_n^2 = \S_n(3412, 3421)$.
For $i = 1, 2$, define
\begin{align*}
	S_i(x,y,z) &= \sum_{n \ge 1} \sum_{\pi \in \S_n^i} 
	x^{\des(\pi)}y^{\dd(\pi)}z^n.
\end{align*}
Moreover, define
\begin{align*}
	F_1(x,y,z) &= \sum_{n \ge 1} \sum_{\pi\in \S_n^1} 
	x^{\des(\pi)}y^{\dd_0(\pi)}z^n \\
	R_1(x,y,z) &= \sum_{n \ge 1} \sum_{\pi \in \S_n^1} 
	x^{\des(\pi)}y^{\dd_\infty(\pi)}z^n \\
	T_2(x,y,z) &= \sum_{n \ge 1} \sum_{\pi \in \S_n^2}
	x^{\des'(\pi)} y^{\dd'(\pi)}.
\end{align*}
We will also use $S_i$, $F_1$, etc. to denote
$S_i(x, y,z)$, $F_1(x, y, z)$, etc.

The proof of the following lemma is very similar to
the proof of \cite[Lemma 3.4]{lin2017gamma}, so it is omitted. The essence of the proof is
Stankova's block decomposition \cite{stankova1994forbidden}.

\begin{lemma}
	We have the system of equations
	\begin{align*}
		S_1 &= z + (z + xyz)S_1 + \frac{2xzS_1^2}{1 - xR_1F_1} + 
		\frac{xzS_1^2(F_1 + xR_1)}{1 - xR_1F_1}, \\
		F_1 &= z + (xzS_1 + zF_1) + \frac{2xzF_1S_1}{1 - xR_1F_1} + 
		\frac{xzF_1S_1(F_1 + xR_1)}{1-xR_1F_1}, \\
		R_1 &= yz + zS_1 + xyzR_1 + \frac{2xzR_1S_1}{1-xR_1F_1} 
		+ \frac{xzR_1S_1(F_1 + xR_1)}{1 - xR_1F_1}.
	\end{align*}
\end{lemma}

Combining the first equation multiplied by $F_1$ and the second equation multiplied by $S_1$, and combining the first equation multiplied by $R_1$ and the third equation multiplied by $S_1$, respectively, gives us 
\[ F_1 = \frac{S_1+xS_1^2}{1+xyS_1}, \quad R_1 = \frac{yS_1 + S_1^2}{1 + S_1}. \]
Plugging these values into the first equation and expanding yields the following.
\begin{corollary}
	We have
	\begin{align*}
		S_1(x,y,z) = xS_1^3(x,y,z) + xzS_1^2(x,y,z) + (z + xyz)S_1(x,y,z) + z.
	\end{align*}
\end{corollary}

We will show that $S_2$ satisfies the same equation.

\begin{lemma}
\label{lem:S2eqns}
	We have the system of equations
	\begin{align*}
		S_2 &= z + zS_2 + (xy - x)zS_2 + xT_2S_2, \\
		T_2 &= z + (x - xy)z^2 + zS_2 + (xyz - 2xz + z)T_2 + xT_2^2.
	\end{align*}
\end{lemma}
\begin{proof}
	By considering the position of $n$, we see
	that every permutation $\pi \in \S_n^2$
	can be uniquely written as either
	$\pi_1 n$ where $\pi_1 \in \S_{n-1}^2$ or	
	$\pi_1 \ast \pi_2$ where $\pi_1\in \S_{k}^2$,
	$\pi_2 \in \S_{n-k}^2$, $1 \le k \le n - 1$, and
	$\pi_1 \ast \pi_2 = AnB$ where
	\begin{align*}
		A&= \pi_1(1) \cdots
		\pi_1(k-1) \\
		B&= (\pi_2(1) + \ell)
		\cdots (\pi_2(j-1) + \ell) \pi_1(k) (\pi_2(j+1) + \ell)
		\cdots (\pi_2(n-k) + \ell),
	\end{align*}
	where $\pi_2(j) = 1$ and $\ell = k - 1$.
	Furthermore,
	\begin{equation*}
	\begin{aligned}[c]
		\des(\pi_1 n) &= \des(\pi_1) \\
		\dd(\pi_1 n) &= \dd(\pi_1) \\
		\des'(\pi_1 n) &= \des(\pi_1)\\
		\dd'(\pi_1 n) &= \dd(\pi_1)
	\end{aligned}
	\hspace{1cm}
	\begin{aligned}
		\des(\pi_1 \ast \pi_2) &= \des'(\pi_1) + \des(\pi_2) + 1 \\
		\dd(\pi_1 \ast \pi_2) &= \dd'(\pi_1) + \dd(\pi_2) \\
		\des'(\pi_1 \ast \pi_2) &= \des'(\pi_1) + \des'(\pi_2) + 1\\
		\dd'(\pi_1 \ast \pi_2) &= \dd'(\pi_1) + \dd'(\pi_2),
	\end{aligned}
	\end{equation*}
	with the exceptions $\dd(1 \ast \pi_2)
	= \dd(\pi_2) + 1$, $\des'(\pi_1 \ast 1) = \des'(\pi_1)$, $\dd'(1 \ast \pi_2) = \dd'(\pi_2) + 1$, and $\des'(1 \ast \pi_2) = \des'(\pi_2)$
	if $n \le 2$.
	With the initial conditions
	\[
		S_2(x, y, z) = z + \cdots,
		\quad T_2(x, y, z) = z + 2z^2 + \cdots,
	\]
	the above implies the stated equations. 
\end{proof}

\begin{proof}[Proof of Theorem \ref{thm:dddes}]
Solving the equations in Lemma \ref{lem:S2eqns} shows
that $S_2$ satisfies the same equation as
$S_1$.
\end{proof}

\section{Concluding remarks and open problems}

Our proofs of the $\gamma$-positivity of $I_n(t)$ and
$J_{2n}(t)$ are purely computational.
Guo and Zeng first suggested the following question.

\begin{prob}[Guo--Zeng \cite{guo2006eulerian}]
	Give a combinatorial interpretation of the 
	coefficients $a_{n,k}$.
\end{prob}

Dilks \cite{dilks2014q} conjectured the following
$q$-analog of the $\gamma$-positivity of 
$I_n(t)$. Here $\maj(\pi)$ denotes the major index
of $\pi$, which is the sum of the descents of $\pi$.

\begin{conjecture}[Dilks \cite{dilks2014q}]
	For $n \ge 1$,
	\[
		\sum_{\pi \in \I_n} t^{\des(\pi)}
		q^{\maj(\pi)} = \sum_{k = 0}^{\lfloor \frac{n-1}{2}\rfloor} \gamma_{n,k}^{(I)} t^k q^{\binom{k+1}{2}} \prod_{i = k + 1}^{n - 1 - k} (1 + tq^i),
	\]
	where $\gamma_{n,k}^{(I)}(q) \in \NN[q]$.
\end{conjecture}

Since the $(3412, 3421)$-avoiding permutations
are invariant under the MFS action, it would be
interesting to find a combinatorial proof of
Theorem \ref{thm:dddes}, since this would
lead to a group action on $\S_n(2413, 3142)$
such that each orbit contains exactly one element
of 
\[
	\{ \pi \in \S_2(2413, 3142) : \dd(\pi) = 0, \des(\pi) = k \}
\]
(cf. \cite[Remark 3.9]{fu2018two}).

\begin{prob}
	Give a bijection between $(3412, 3421)$-avoiding
	permutations with $m$ double descents and $k$
	descents and separable permutations with
	$m$ double descents and $k$ descents.
\end{prob}

Note that there does not exist a bijection preserving
descent sets because the separable permutations
are not $\DES$-Wilf equivalent to any permutation
classes avoiding two patterns.

Finally, Lin \cite{lin2017gamma} proved that the only
permutations $\sigma$ of length $4$ which satisfy
\[
	\sum_{\pi \in \S_n(\sigma, \sigma^r)} t^{\des(\pi)}
	= \sum_{k = 0}^{\lfloor\frac{n-1}{2}\rfloor}
	\gamma_{n,k}t^k(1 + t)^{n-1-2k}
\]
where
\[
	\gamma_{n,k} = \#\{ \pi \in \S_n(\sigma, \sigma^r)
	: \dd(\pi) = 0, \des(\pi) = k\}
\]
are the permutations $\sigma = 2413$, $3142$, $1342$, $2431$.
Here $\sigma^r$ denotes the reverse of $\sigma$.
We can similarly ask the following.

\begin{prob}
\label{prob:reverse}
	Which permutations $\sigma$ of length $\ell \ge 6$
	satisfy the above property?
\end{prob}

\begin{remark}
For $\ell = 5$, the answer to Problem \ref{prob:reverse} is $\sigma = 13254$, $15243$, $15342$, $23154$, $25143$ and their reverses. We have verified using Sage that these are the only permutations which satisfy the property for $n = 5, 6, 7$, and these permutation classes are all invariant under the MFS action because in these patterns, every index $i \in [5]$ is either a valley or a peak.
\end{remark}

\section{Acknowledgements}
This research was conducted at the University of Minnesota
Duluth REU and was supported by NSF / DMS grant 1650947 and 
NSA grant H98230-18-1-0010. I would like to thank
Joe Gallian for suggesting the problem, and Brice Huang
for many careful comments on the paper.

\bibliography{duluth}{}
\bibliographystyle{acm}
\end{document}